\documentclass[11pt]{amsart}
\usepackage{amsmath,mathrsfs}

\begin{document}
\theoremstyle{theorem}
\newtheorem{thm}{Theorem}
\newtheorem{lem}{Lemma}
\newtheorem{prop}{Proposition}
\newtheorem{cor}{Corollary}
\theoremstyle{remark}
\newtheorem{rem}{Remark}
\newtheorem*{ack}{Acknowledgement}

\def\Z{{\mathbb Z}} 
\def\C{{\mathbb C}} 
\def\R{{\mathbb R}} 

\def\A{{\mathfrak A}}
\def\B{{\mathcal B}}
\def\H{{\mathcal H}}

\def\L{{\mathfrak L}}
\def\LH{{\mathfrak L(\H)}}
\def\ran{\mathrm{ran}}

\mathchardef\ordinarycolon\mathcode`\: 
\def\vcentcolon{\mathrel{\mathop\ordinarycolon}} 
\providecommand*\coloneqq{\mathrel{\vcentcolon\mkern-1.2mu}=}

\title[A simple proof of the Fredholm Alternative]{A simple proof of the Fredholm Alternative}
\author{Otgonbayar Uuye}
\address{University of Copenhagen\\Universitetsparken 5\\DK-2100 Copenhagen E\\Denmark}
\date{\today}
\begin{abstract}
In this expository note, we present a simple proof of the Fredholm Alternative for compact operators that are norm limits of finite rank operators.
\end{abstract}

\maketitle

We recall that a bounded linear operator $T: E \to E$ on a Banach space $E$ is called a {\em compact} operator if $T$ maps the closed unit ball of $E$ to a relatively compact subset of $E$. The following is a basic result about compact operators known as the {\em Fredholm Alternative}. See for instance \cite{MR1157815}.
\begin{thm}\label{thm FA} Let $T$ be a compact operator on a Banach space $E$ and let $I$ denote the identity operator on $E$. Then either
\begin{enumerate}
\item[(i)] the operator $I - T$ is invertible (i.e.\ has a bounded inverse), or 
\item[(ii)] there exists a nonzero vector $y \in E$, $y \ne 0$, such that $Ty = y$.
\end{enumerate}
\end{thm}

In this note, we give a simple proof of the following version (Theorem~\ref{thm alt FA}) of the Fredholm Alternative.

 A {\em finite-rank} operator is a bounded linear operator whose range is  finite dimensional. 

\begin{thm}\label{thm alt FA} Let $T$ be a bounded linear operator on a Banach space $E$. Suppose that there exists a finite-rank operator $F$ such that $||T-F|| < 1$. Then either
\begin{enumerate}
\item[(i)] the operator $I - T$ is invertible, or 
\item[(ii)] there exists a nonzero vector $y \in E$, $y \ne 0$, such that $Ty = y$.
\end{enumerate}
\end{thm}

We start with a simple lemma.
\begin{lem}\label{lem rk} Let $S: V \to V $ be a linear operator on a vector space $V$ with range $M \subseteq V$. Let $I$ denote the identity operator on $V$. Then the operator $I - S: V \to V$ restricts to a linear operator $(I - S)_{|M}: M \to M$. Consider the following statements: 
\begin{enumerate}
\item\label{IV} The operator $I - S: V \to V$ is injective.
\item\label{IM} The operator $(I-S)_{|M}: M \to M$ is injective.  
\item\label{SV} The operator $I - S: V \to V$ is surjective.
\item\label{SM} The operator $(I-S)_{|M}: M \to M$ is surjective.
\end{enumerate}
Then (\ref{IV}) $\Leftrightarrow$ (\ref{IM}), and (\ref{SV}) $\Leftrightarrow$ (\ref{SM}). 
\end{lem}

\begin{proof}
It is clear that $I-S$ maps $M$ into $M$.

The implication (\ref{IV}) $\Rightarrow$ (\ref{IM}) is also clear. 

Suppose that $(I-S)y = 0$ for some $y$ in $V$. Then $y = Sy$, hence $y$ belongs to $M$. Thus (\ref{IM}) implies (\ref{IV}).
 
Let $x$ be a vector in $M$ and suppose that $(I -S)y = x$ for some $y$ in $V$. Then $y = x + Sy$, hence $y$ belongs to $M$. Thus (\ref{SV}) implies (\ref{SM}).
 
Conversely, let $x$ be a vector in $V$ and suppose that $(I-S)y = Sx$ for some $y$ in $M$. Then $(I-S)(x + y) = x$. Hence (\ref{SM}) implies (\ref{SV}). 
\end{proof}

\begin{cor}\label{cor R-F} Let $R: V \to V$ be a linear operator on a vector space $V$.  Suppose that $R$ is bijective. Then for any finite-rank operator $F$, the operator $R - F$ is injective if and only if $R - F$ is surjective. 
\end{cor}
\begin{proof}
Let $S: V \to V$ denote the operator $F \circ R^{-1}$ and let $M$ denote the range of $S$. Then $M$ is finite dimensional, and thus $(I-S)_{|M}: M \to M$ is injective if and only if $(I-S)_{|M}: M \to M$ is surjective. It follows from Lemma~\ref{lem rk} that $I-S$ is injective if and only if $I-S$ is surjective.
However, since $R$ is bijective, it follows from the identity $$R - F = (I - S) \circ R$$ that $R - F$ is injective if and only if $I - S$ is injective, and $R - F$ is surjective if and only if $I - S$ is surjective. This completes the proof.
\end{proof}

\begin{prop}\label{prop I - T} Let $T: V \to V$ be a linear operator on a vector space $V$. Suppose that there exists a finite-rank operator $F$ such that the operator $I - (T - F): V \to V$ is bijective. Then $I - T$ is injective if and only if $I - T$ is surjective. 
\end{prop}
\begin{proof} Let $F$ be a finite rank operator such that $R \coloneqq I - (T -F)$ is bijective. Then $I - T = R - F$ and the proposition follows from Corollary~\ref{cor R-F}.
\end{proof}

\begin{proof}[Proof of Theorem~\ref{thm alt FA}] Suppose that (ii) does {\em not} hold. Then the operator $I - T$ is injective. Let $F$ be a finite-rank operator such that $||T -F|| < 1$. Then the operator $I -(T -F)$ is invertible, in particular, bijective. Hence by Proposition~\ref{prop I - T}, the operator $I - T$ is also surjective. By the Open Mapping Theorem, $I - T$ is invertible, i.e.\ (i) holds.
\end{proof}

Recall that the norm limit of (a net of) finite-rank operators is always compact. We say that a Banach space $E$ has the {\em approximation property} if every compact operator is a norm limit of finite-rank operators. Theorem~\ref{thm alt FA} proves the Fredholm Alternative (Theorem~\ref{thm FA}) for Banach spaces with the approximation property. Most familiar Banach spaces have the approximation property.\footnote{but not all -- see \cite{MR0402468}.}

As an example, we prove that Hilbert spaces have the approximation property. This proves the Fredholm Alternative for compact operators on a Hilbert space.  
\begin{lem}\label{lem BAP to AP} Let $E$ be a Banach space and let $P_{\alpha}$ be a net of uniformly bounded linear operators strongly converging to $I$. Then for any compact operator $T$, the net $P_{\alpha}T$ converges to $T$ in norm.
\end{lem}
\begin{proof}
Let $B$ denote the closed unit ball of $E$. Then, since $T$ is compact, the subset $T(B)$ is relatively compact, hence totally bounded in $E$, and for any $\alpha$, 
	\begin{align}
	||T-P_{\alpha}T|| &= \sup_{x \in B} ||(T-P_{\alpha}T)x||\\ 
	&= \sup_{y \in T(B)} ||(I - P_{\alpha})y||.
	\end{align}
By assumption, $I - P_{\alpha}$ converges pointwise to the $0$ operator. Moreover, since each $P_{\alpha}$ is linear and $\{P_{\alpha}\}$ uniformly bounded, the net $I - P_{\alpha}: E \to E$  is uniformly equicontinuous, hence converges uniformly on $T(B)$.	
\end{proof}

We say that a Banach space $E$ has the {\em bounded approximation property} if there exists a net of uniformly bounded finite-rank operators on $E$ converging to $I$ strongly. It follows from Lemma~\ref{lem BAP to AP}, that if $E$ has the bounded approximation property, then it has the approximation property (this result is due to Grothendieck). Finally, it follows from Parseval's identity, that Hilbert spaces have the bounded approximation property. This concludes our proof of the Fredholm Alternative for compact operators on a Hilbert space.

As an application, we prove a Fredholm Alternative for pseudodifferential operators of negative order.
\begin{thm} Let $X$ be a closed, smooth manifold and let \[P: C^{\infty}(X) \to C^{\infty}(X)\] be a pseudodifferential operator of order $m < 0$. Then $I -P$ is injective if and only if $I - P$ is surjective.
\end{thm}
\begin{proof} 

Fix a smooth measure on $X$ and let $H = L^{2}(X)$. Then $P$ extends to a compact operator  $\bar P: H \to H$ and by elliptic regularity if $x \in C^{\infty}$ and $(I-\bar P)y = x$, then $y$ belongs to $C^{\infty}$ (cf.\ \cite{MR1852334}).

It follows that $I - P$ is injective if and only if $I - \bar P$ is injective, and $I - P$ is surjective if and only if $I - \bar P$ is surjective. But, by the Fredholm Alternative for compact operators on a Hilbert space, $I - \bar P$ is injective if and only $I- \bar P$ is surjective.


\end{proof}

\begin{rem}
We note that elliptic regularity is especially simple in our case. Indeed, for $s \ge 0$, let $H^{s}$ denote the $s$-Sobolev space. Then $H^{0} = L^{2}(X)$ and $\bigcap_{s \ge 0} H^{s} = C^{\infty}(X)$ and $\bar P (H^{s}) \subseteq H^{s - m}$ for $s \ge 0$.

Elliptic Regularity: If $x \in H^{s - m}$ and $(I - \bar P)y = x$ for some $y \in H^{s}$, then $y = x + \bar P y$ belongs to $H^{s - m}$. 
\end{rem}

\begin{rem} Since $C^{\infty}(X) = \bigcap_{s \ge 0} H^{s}$ is a Fr\'{e}chet space and $P$ is continuous, if $I-P$ is bijective, then it has a continuous inverse by the Open Mapping Theorem. Hence, either
\begin{enumerate}
\item[(i)] the operator $I - P: C^{\infty}(X) \to C^{\infty}(X)$ is invertible, or 
\item[(ii)] there exists a nonzero $y \in C^{\infty}(X)$, $y \ne 0$, such that $Py = y$.
\end{enumerate}
\end{rem}

\begin{ack} The author gratefully acknowledges support from the Centre for Symmetry and Deformation at the University of Copenhagen and the Danish National Research Council.
\end{ack}

\bibliographystyle{amsalpha}
\bibliography{../BibTeX/biblio}

\def\romsup#1{{\edef\next{\the\font}$^{\next#1}$}} \def\cprime{$'$}
  \def\cftil#1{\ifmmode\setbox7\hbox{$\accent"5E#1$}\else
  \setbox7\hbox{\accent"5E#1}\penalty 10000\relax\fi\raise 1\ht7
  \hbox{\lower1.15ex\hbox to 1\wd7{\hss\accent"7E\hss}}\penalty 10000
  \hskip-1\wd7\penalty 10000\box7} \def\Dbar{\leavevmode\lower.6ex\hbox to
  0pt{\hskip-.23ex \accent"16\hss}D}
  \def\cfac#1{\ifmmode\setbox7\hbox{$\accent"5E#1$}\else
  \setbox7\hbox{\accent"5E#1}\penalty 10000\relax\fi\raise 1\ht7
  \hbox{\lower1.15ex\hbox to 1\wd7{\hss\accent"13\hss}}\penalty 10000
  \hskip-1\wd7\penalty
  10000\box7}\def\cfudot#1{\ifmmode\setbox7\hbox{$\accent"5E#1$}\else
  \setbox7\hbox{\accent"5E#1}\penalty 10000\relax\fi\raise 1\ht7
  \hbox{\raise.1ex\hbox to 1\wd7{\hss.\hss}}\penalty 10000 \hskip-1\wd7\penalty
  10000\box7} \def\polhk#1{\setbox0=\hbox{#1}{\ooalign{\hidewidth
  \lower1.5ex\hbox{`}\hidewidth\crcr\unhbox0}}}
  \def\cydot{\leavevmode\raise.4ex\hbox{.}}
  \def\cfgrv#1{\ifmmode\setbox7\hbox{$\accent"5E#1$}\else
  \setbox7\hbox{\accent"5E#1}\penalty 10000\relax\fi\raise 1\ht7
  \hbox{\lower1.05ex\hbox to 1\wd7{\hss\accent"12\hss}}\penalty 10000
  \hskip-1\wd7\penalty 10000\box7}
  \def\uarc#1{\ifmmode{\lineskiplimit=0pt\oalign{$#1$\crcr
  \hidewidth\setbox0=\hbox{\lower1ex\hbox{{\rm\char"15}}}\dp0=0pt
  \box0\hidewidth}}\else{\lineskiplimit=0pt\oalign{#1\crcr
  \hidewidth\setbox0=\hbox{\lower1ex\hbox{{\rm\char"15}}}\dp0=0pt
  \box0\hidewidth}}\relax\fi}
\providecommand{\bysame}{\leavevmode\hbox to3em{\hrulefill}\thinspace}
\providecommand{\MR}{\relax\ifhmode\unskip\space\fi MR }
\providecommand{\MRhref}[2]{%
  \href{http://www.ams.org/mathscinet-getitem?mr=#1}{#2}
}
\providecommand{\href}[2]{#2}
\begin{thebibliography}{Rud91}

\bibitem[Enf73]{MR0402468}
Per Enflo, \emph{A counterexample to the approximation problem in {B}anach
  spaces}, Acta Math. \textbf{130} (1973), 309--317. \MR{0402468 (53 \#6288)}

\bibitem[Rud91]{MR1157815}
Walter Rudin, \emph{Functional analysis}, second ed., International Series in
  Pure and Applied Mathematics, McGraw-Hill Inc., New York, 1991. \MR{MR1157815
  (92k:46001)}

\bibitem[Shu01]{MR1852334}
M.~A. Shubin, \emph{Pseudodifferential operators and spectral theory}, second
  ed., Springer-Verlag, Berlin, 2001, Translated from the 1978 Russian original
  by Stig I. Andersson. \MR{MR1852334 (2002d:47073)}

\end{thebibliography}

\end{document}